\documentclass{amsart}
\usepackage{amssymb,amsmath,amsthm,amsfonts, amsaddr}
\usepackage{mathtools}

\usepackage{enumitem}
\setlist[enumerate]{}

\usepackage{xcolor} % package crucial to implement colors  
\usepackage{tikz} % package used for the tikz
\usepackage{tikz-cd}

\newtheorem{theorem}[subsection]{Theorem}
\newtheorem{defn}[subsection]{Definition}
\newtheorem{lemma}[subsection]{Lemma}
\newtheorem{cor}[subsection]{Corollary}

\theoremstyle{definition}

\newtheorem{rmk}[subsection]{Remark}
\title{A partition of finite rings makes lifting possible}

\author{Vineeth Chintala}
\address{\small{}}

\email{vineethreddy90@gmail.com}
\thanks{ The author is supported by the DST-Inspire faculty fellowship in India.}

\begin{document}

\begin{abstract}
We show that every finite ring has a partition, where each block corresponds to one idempotent. Remarkably, this partition provides a way to \emph{lift} a wide variety of special elements such as idempotents, nilpotents, unipotents, roots of unity and regular elements.
\end{abstract}

\maketitle

\section{Idempotent Partitions}
We show that every finite ring has a partition where each block corresponds to one idempotent. This partition provides a way to \emph{lift} a wide variety of special elements such as idempotents, nilpotents, unipotents, roots of unity and regular elements.

\begin{rmk}Throughout the paper, we assume that $R$ is a finite associative ring. The interested reader can easily check that most of the results extend (with exactly the same proofs) to finite power-associative rings.
\end{rmk}

An element $e$ is said to be an idempotent if $e^2 =e$. Idempotents play a central role in the classification of algebras, and the partition indicates that they may also have some combinatorial significance.  

\begin{lemma}\label{lemma1}
Let $R$ be a finite  ring and $x \in R$. 
\begin{enumerate}[label=\alph*.]
\item Then $e_x = x^n$ is an idempotent for some positive integer $n$. 
\item Further, the idempotent $e_x$ is uniquely determined by $x \in R$.\qedhere 
\end{enumerate}
\end{lemma}
\begin{proof}
The set $\{x^{2^i} :  i \geq 1\}$ is finite and so $x^{2^r} = (x^{2^r})^t$ for some integers $r,t > 1$. 
\vskip 1mm
Let $y = x^{2^r}$. Then $y^{t-1}$ is an idempotent. Indeed, 
\[ y^{2(t-1)} = y^{t-2}y^t = y^{t-2}y = y^{t-1}. \]
To prove uniqueness, suppose $x^r = e_1$ and $x^s = e_2$. Then
\[ e_1 = (x^r)^s = (x^s)^r = e_2. \qedhere \]
\end{proof}

\begin{defn}
For each idempotent $e\in R$, define \[B_e = \{x \in R : x^k = e \text{ for some positive integer } k\}. \qedhere\] 
\end{defn}
\vskip 2mm
Put every element $x$ in the corresponding block $B_{e_x}$. 
This gives a partition of the ring into blocks, where each block corresponds to one idempotent. 
   \[ R = B_0 \sqcup B_1 \sqcup \cdots \]

 Here we are talking about partitions of rings as sets. For such partitions to be useful, they should be compatible with ring homomorphisms. Indeed, this compatibility is easy to check.

\begin{theorem}\label{compatible}
Let $\phi : R\rightarrow R'$ be a homomorphism  between finite rings. 
Then the map $\phi$ preserves their idempotent partitions 
\[ \begin{tikzcd}
R \arrow{r}{\phi} \arrow[swap]{d}{} & R' \arrow{d}{} \\%
\{B_e\}\arrow{r}{\phi}& \{B'_{e'}\}
\end{tikzcd}
\]
where $\{B_e\}$ and $\{B'_{e'}\}$ denote the idempotent-partitions of $R, R'$ respectively.
\end{theorem}
\begin{proof}
Let $a \in B_e$. Since $\phi(a^n) = \phi(a)^n$, we have $\phi(a) \in B'_{\phi(e)}$. In other words, \[\phi(B_e) \subseteq B'_{\phi(e)}. \qedhere \] 
\end{proof}
\vskip 3mm

\begin{rmk}In fact, when $\phi$ is surjective, we have 
\[B'_{e'} = \{\bigcup_e \phi(B_e) : \phi(e) = e'\}.\] 
Notice that the number of blocks in $\phi(R)$ is at most the number of blocks in $R$.
\end{rmk}

\vskip 5mm
\section{Lifting of special elements}
\vskip 5mm

The idempotent partition of the ring opens the door for the lifting of many types of elements.
\vskip 2mm
\begin{defn}Given any map $\phi : R \rightarrow S$, we say that idempotents can be lifted (over $\phi$) if every idempotent in $S$ has a preimage which is also an idempotent. The lifting of other special elements is similarly defined.
\end{defn}

\vskip 2mm

\begin{theorem}\label{lifting}
Let $\phi : R \rightarrow S$ be a surjective homomorphism between two finite power-associative rings. Then the following types of elements can \textbf{always} be lifted over $\phi$.
\begin{enumerate}[label=\alph*.]
\item Idempotents
\item Nilpotents
\item Unipotents
\item Roots of unity
\qedhere
\end{enumerate}
\end{theorem}
\begin{proof}
Let $\bar{x} = \phi(x)$ for $x \in R$, and $e = x^n$ be the \emph{idempotent} corresponding to $x$ (see Lemma~\ref{lemma1}).  The goal, in each of the below cases, is to find a pre-image with the same property as $\bar{x}$.
\begin{enumerate}[label=\alph*.]
    \item Suppose $\bar{x} \in S$ is an idempotent. Then $\phi(e) = \bar{x}^n = \bar{x}$ is a solution.
    \vskip 2mm
    \item Suppose $\bar{x} \in S$ is a nilpotent. Then we have $\bar{e} = 0$. 
    \vskip 1mm
     \noindent Now consider $z = x(1-e)$. Note that $\bar{z} = \bar{x}$ since $\bar{e} =0$. Since $z^n = x^n(1-e)^n = e(1-e)= 0$, the element $z$ is a solution.
    \vskip 2mm
    \item Suppose $\bar{x} \in S$ is a unipotent; In other words $\bar{1} - \bar{x}$ is a nilpotent which (we just proved) lifts to a nilpotent $z= (1-y)$ for some $y\in R$. Clearly $\bar{y} = \bar{x}$, so $y$ is a solution.
        \vskip 2mm
    \item Suppose $\bar{x}^m = \bar{1}$. Consider $y= xe + 1-e$. Note that $\bar{y} = \bar{x}$ as $\bar{e} = \bar{1}$. Moreover since $e(1-e)$ = 0, we have 
    \[ y ^k = (xe + 1-e)^k = (xe)^k + (1-e) \] for all positive integers $k$. In particular, $y^n = e + 1-e =1.$ \qedhere
\end{enumerate}
\end{proof}

\begin{cor}
Let $\phi : R \rightarrow S$ be a surjective homomorphism between two finite rings. Let $idem(R), idem(S)$ denote the set of idempotents of the respective rings $R, S$. Then 
\[|idem(R)| \geq |idem(S)|.\]
\end{cor}

\begin{rmk}For associative rings, it is known that idempotents in $R/I$ lift to $R$ when $I$ is a nil ideal or when $R$ is $I$-$adically$ complete (\cite{L}, Theorems 21.28, 21.31). Neither of these conditions hold when $I$ contains an idempotent. The paper \cite{DDN} gives a few counterexamples where idempotents cannot always be lifted in infinite associative rings.
\end{rmk}
%%%

The lifting of idempotents forces the lifting of other properties. (See \cite{KLN} for some consequences of idempotent lifting). We will now see that \emph{regular} elements can also be lifted. 

\begin{defn}An element $x$ is said to be regular (or \emph{von Neumann regular}) if $xyx =x $ for some element $y \in R$. Notice that every idempotent $e$ is regular as $e^2  = e$ (take $x = e, y =1$). 
\end{defn}

In general, lifting of idempotents does not imply lifting of regular elements. For a counterexample, consider the map 
$\phi : \mathbb{Z} \rightarrow \mathbb{Z}/{8\mathbb{Z}}$. Here $\mathbb{Z}/{8\mathbb{Z}}$ has only trivial idempotents $\{\overline{0},\overline{1}\}$ which can obviously be lifted, but some of its regular elements $\{ \overline{3}, \overline{5} \}$ cannot be lifted to regular elements in $\mathbb{Z}$. However, it was shown that if idempotents lift modulo \emph{every} left ideal contained in a two-sided ideal $I$, then regular elements lift modulo $I$ (\cite{KL}, Theorem 9.3). Unsurprisingly the proof is a bit simpler for finite rings.

\vskip 2mm
The following result shows that regular elements can \emph{always} be lifted for finite rings.

\begin{theorem}
Let $R$ be a finite associative ring and $I$ be a left ideal. Then regular elements in $R/I$ can be lifted to regular elements in $R$.
\end{theorem} 
\begin{proof}
Let $x,y$ be two elements such that $xyx -x \in I$. Since $R$ is a finite ring, $e = (xy)^n$ is an idempotent for some positive integer $n$.
We need to show that there is a regular element $z \in R$ such that $z -x \in I$.
\vskip 1mm
Take $z = (xy)^{2n-1} x$. Then $zy = e$ and $ez =z$. Therefore
\[ zyz = ez = z.\]
Further, \[z-x =  ((xy)^{2n-1} - 1)x =  r (xy-1)x \] for some $r \in R$.
Since $(xy-1) x = xyx -x  \in I$, we have $z -x \in I$. 
\end{proof}

\vskip 5mm
\section{Zero Divisors in associative rings}

\vskip 3mm
\begin{theorem}\label{block}
Let $R$ be a finite associative ring with partition $\{B_e : e \in Idem(R)\}$. The blocks $B_e$ satisfy the following properties.
  \begin{enumerate}[label=\alph*.]
  \item Let $x_i \in B_{e_i}$. If $x_1x_2 =0$, then $e_1e_2 = 0$.
  \item Let $x,y \in B_e$. If $xy =0$, then $e=0$.
   \item $B_0$ consists of all nilpotent elements and $B_1$ consists of all invertible elements in $R$.\qedhere
  \end{enumerate}
\end{theorem}
\begin{proof}
Suppose $x_1^r = e_1$ and $x_2^s = e_2$.
Then $x_1x_2= 0$ implies that \[e_1e_2 = x_1^rx_2^s = 0. \] 
Taking $e_1 =e_2$, the second statement obviously follows.

If an element $x$ is invertible, then so is $e_x$. Since $e_x^2 = e_x$ it follows that $e_x =1$. Finally, note that any element $x \in B_0$ satisfies $x^n =0$ for some positive integer $n$.
\end{proof}

Since $e_x^2 = e_x$, the element $e_x$ is a zero divisor if $e_x \neq 1$. In that case $x$ will also be a zero divisor. Therefore $\{ B_e |e\neq 1\}$ is a partition of the zero-divisors of $R$.

\begin{defn}
 Following \cite{B} one can consider any subset $S\subseteq R$ as a directed graph, where there is an edge $x \longrightarrow y$ between two elements $x, y$ if and only if $xy=0$.  We'll refer to this graph as the zero-divisor graph $\Gamma(S)$. 
\end{defn}

Suppose there is an edge $x \longrightarrow y$ between two elements $x, y \in \Gamma(R)$. Then it follows (from Theorem$~\ref{block}$) that there is also an edge $e_x \longrightarrow e_y$. In particular, if $e \neq 0$ then there are no edges between elements inside $B_e$.

\begin{theorem}
Let $R$ be a finite associative ring. Then $\{B_e : e \neq 0\}$ is a partition of the subgraph
$\Gamma(R \setminus B_0)$. 
\end{theorem}

A subset $\{x_1, \cdots, x_n\}$ is called a clique if $x_i \longrightarrow x_j$ and $x_j \longrightarrow x_i$  for all $i \neq j$.

\begin{theorem}
If $\{e_1, \cdots ,e_n\}$ is a maximal clique of non-zero idempotents in $\Gamma(R)$, then $\sum\limits_{i=1}^n e_i =1.$
\end{theorem}
\begin{proof}
Let $d = 1 - \sum\limits_{i=1}^n e_i$. Then $d$ is also an idempotent and $de_i =e_id = 0$. Therefore $\{d, e_1, \cdots, e_n\}$is a larger clique unless $d=0$.
\end{proof}

%%--------------------Here the manuscript ends--------------------------------

\begin{thebibliography}{00}

\bibitem{B} I. Beck, Coloring of commutative rings, \textit{Journal of Algebra} (1988), 116 (1): 208--226.
\bibitem{DDN} Alexander J. Diesl, Samuel J. Dittmer, and Pace P. Nielsen: Idempotent lifting and ring extensions. \textit{J. Algebra Appl.} 15(6):1650112 (16 pages), 2016.
\bibitem{KL} Dinesh Khurana and T. Y. Lam, Rings with internal cancellation, J. Algebra 284 (2005), no. 1, 203--235. 
\bibitem{KLN} Dinesh Khurana, T. Y. Lam, and Pace P. Nielsen : An ensemble of idempotent lifting hypotheses. \textit{J. Pure Appl. Algebra} 222(6):1489--1511, 2018.
\bibitem{L} T. Y. Lam, A First Course in Noncommutative Rings, second ed., Graduate Texts in Mathematics, vol. 131, \textit{Springer-Verlag, New York}, 2001

\end{thebibliography}
\end{document}